\documentclass[11pt]{amsart}

\usepackage{a4wide}
\usepackage{amssymb}
\usepackage{mathrsfs}
\usepackage{enumerate}
\usepackage{esint}
\usepackage{titletoc}
\usepackage[colorlinks=true, urlcolor=blue, linkcolor=blue, citecolor=black]{hyperref}
\usepackage[nameinlink]{cleveref}
\usepackage{graphicx}
\usepackage{todonotes}

\theoremstyle{plain}
\newtheorem{theorem}{Theorem}[section]
\newtheorem{lemma}[theorem]{Lemma}
\newtheorem{corollary}[theorem]{Corollary}
\newtheorem{proposition}[theorem]{Proposition}

\theoremstyle{definition}
\newtheorem{definition}[theorem]{Definition}
\newtheorem{remark}[theorem]{Remark}
\newtheorem*{remark*}{Remark}

\numberwithin{equation}{section}

\renewcommand{\d}{\textnormal{d}}
\newcommand{\dx} {\,\mathrm{d}x}
\newcommand{\dy} {\,\mathrm{d}y}

\newcommand{\R}{\mathbb{R}}
\newcommand{\Rd}{\mathbb{R}^d}

\newcommand{\sign}{\text{sign}}

\newcommand{\eps}{\varepsilon}

\makeatletter
\@namedef{subjclassname@2020}{\textup{2020} Mathematics Subject Classification}
\makeatother

\title[Harnack inequality for nonlocal problems with non-standard growth]{Harnack inequality for nonlocal problems with non-standard growth}

\author{Jamil Chaker}
\address{Fakult\"at f\"ur Mathematik, Universit\"at Bielefeld, 33615 Bielefeld, Germany}
\email{jchaker@math.uni-bielefeld.de}

\author{Minhyun Kim}
\address{Fakult\"at f\"ur Mathematik, Universit\"at Bielefeld, 33615 Bielefeld, Germany}
\email{minhyun.kim@uni-bielefeld.de}

\author{Marvin Weidner}
\address{Fakult\"at f\"ur Mathematik, Universit\"at Bielefeld, 33615 Bielefeld, Germany}
\email{mweidner@math.uni-bielefeld.de}


\thanks{Jamil Chaker gratefully acknowledges funding by the Deutsche Forschungsgemeinschaft (SFB 1283/2 2021 - 317210226). Minhyun Kim and Marvin Weidner gratefully acknowledge funding by the Deutsche Forschungsgemeinschaft (GRK 2235/2 2021 - 282638148).}

\begin{document}

\begin{abstract}
We prove a full Harnack inequality for local minimizers, as well as weak solutions to nonlocal problems with non-standard growth.
The main auxiliary results are local boundedness and a weak Harnack inequality for functions in a corresponding De Giorgi class.
This paper builds upon a recent work on regularity estimates for such nonlocal problems by the same authors.
\end{abstract}

\maketitle

\section{Introduction} \label{sec:introduction}
The goal of the present work is to prove a full Harnack inequality for local minimizers and weak solutions to a class of nonlocal problems which exhibit non-standard growth. 
This article builds upon the recent paper \cite{ChKiWe21}, in which we study regularity properties for local minimizers of nonlocal energy functionals, as well as weak solutions to nonlocal equations with non-standard growth. We prove that these objects satisfy a suitable fractional Caccioppoli inequality and therefore belong to corresponding De Giorgi classes. In this work, we show that any function in such De Giorgi class satisfies a full Harnack inequality.
As a consequence, we obtain the full Harnack inequality for local minimizers and weak solutions. 

Before we state the main result of this paper, let us formulate the main assumptions and briefly present the energy functionals respectively the nonlocal operators considered in this work. We point out that the setup of this article is in align with \cite{ChKiWe21}.

Let $\Omega \subset \R^d$ be open, $s \in (0,1)$ and $1 \leq p \leq q$. 
Throughout the paper, let $f: [0, \infty) \to [0, \infty)$ be convex, strictly increasing and differentiable with $f(0)=0$ and $f(1)=1$.
We say that $f$ satisfies \eqref{eq:pq} if for all $t \geq 0$:
\begin{subequations}
\makeatletter
\def\@currentlabel{$f_{p}^{q}$}
\makeatother
\label{eq:pq}
\begin{align}
p f(t) \leq ~&t f'(t), \label{eq:pq-lower} \tag{$f_{p}$} \\
& t f'(t) \le qf(t). \label{eq:pq-upper} \tag{$f^{q}$}
\end{align}
\end{subequations}
The growth function $f$ is naturally associated with nonlocal energy functionals and nonlocal operators. On the one hand, consider
\begin{equation} \label{eq:nonlocalfunctional}
u \mapsto \mathcal{I}_f(u) = (1-s)\iint_{(\Omega^c \times \Omega^c)^c} f\left( \frac{\vert u(x) - u(y) \vert}{\vert x-y \vert^s} \right) \frac{k(x,y)}{|x-y|^{d}} \dy \dx,
\end{equation}
where $k: \Rd \times \Rd \to \R$ is a measurable function satisfying
\begin{equation} \label{eq:k} \tag{$k$}
k(x,y) = k(y,x) \quad\text{and}\quad \Lambda^{-1} \leq k(x,y) \leq \Lambda \quad\text{for a.e. } x, y \in \Rd
\end{equation}
for some $\Lambda \ge 1$. 
In \cite[Theorem 6.2]{ChKiWe21}, we prove that local minimizers of $\mathcal{I}_f$ belong to the De Giorgi class $G(\Omega; q, c, s, f)$ for some constant $c = c(d,q,\Lambda) > 0$ if $f$ satisfies \eqref{eq:pq-upper} for some $q > 1$. For the precise definition of the De Giorgi class, see \Cref{def:DG}. 

On the other hand, we consider weak solutions to 
\begin{equation} \label{eq:nonlocalequation}
\mathcal{L}_{h}u = 0 \quad\text{in } \Omega,
\end{equation}
where $\mathcal{L}_{h}$ is a nonlocal operator of the form
\begin{equation*}
\mathcal{L}_{h}u(x) = (1-s) \text{p.v.} \int_{\Rd} h\left( x, y, \frac{u(x)-u(y)}{|x-y|^s}\right) \frac{\dy}{|x-y|^{d+s}}.
\end{equation*}
Here, $h : \Rd \times \Rd \times \R \to \R$ is a measurable function satisfying the structure condition
\begin{equation}\label{eq:h} \tag{$h$}
h(x, y, t) = h(y, x, t), \quad \sign(t)\frac{1}{\Lambda} f'(|t|) \leq h(x,y,t) \leq \Lambda f'(|t|)
\end{equation}
for a.e. $x, y \in \Rd$ and for all $t \in \R$. We show in \cite[Theorem 7.3]{ChKiWe21} that weak solutions to \eqref{eq:nonlocalequation} are in $G(\Omega; q, c, s, f)$ for some constant $c = c(d,q,\Lambda) > 0$ if \eqref{eq:pq-upper} holds true for $q > 1$.

The goal of this article is to prove a full Harnack inequality of the following form.
\begin{theorem}
\label{thm:harnack}
Let $\Omega$ be an open subset in $\Rd$. Let $0 < s_0 \leq s < 1$, $1 < p \leq q$, $c > 0$ and assume that $f$ satisfies \eqref{eq:pq}. There exists a constant $C > 0$ such that if $u \in G(\Omega; q, c, s, f)$ is nonnegative in $B_{R}(x_0) \subset \Omega$, then
\begin{equation}
\label{eq:harnack}
\sup_{B_{R/2}(x_0)} u \leq C \left( \inf_{B_{R/2}(x_0)} u + \mathrm{Tail}_{f'}(u_-; x_0, R) \right).
\end{equation}
The constant $C$ depends only on $d$, $s_0$, $p$, $q$ and $c$.
\end{theorem}
The Harnack inequality was originally proved for harmonic functions and later obtained for several elliptic and parabolic local operators. It is known to have important consequences such as a priori estimates in H\"older spaces or convergence theorems. 
Therefore, it plays an important role in several mathematical fields such as geometric analysis, probability or analysis of partial differential equations. For an introduction to Harnack inequalities, their history and consequences, we refer the reader to the article by Kassmann \cite{KassmannHarnack}. 
The appearance of the tail term on the right-hand side of \eqref{eq:harnack} is a purely nonlocal phenomenon.
It is shown in \cite{Ka07} that the classical Harnack inequality fails for $s$-harmonic functions if nonnegativity of the function is assumed in the solution domain only. In \cite{Ka11}, a new formulation of the Harnack inequality is introduced. It involves a nonlocal tail term as in \eqref{eq:harnack} which captures the negative values of the $s$-harmonic function outside the solution domain.
In our setup the nonlocal tail is of the following form
\[ \mathrm{Tail}_{f'}(u; x_0, R) = R^s (f')^{-1} \left( (1-s)R^s \int_{\Rd \setminus B_R(x_0)} f' \left( \frac{|u(y)|}{|y-x_0|^s} \right) \frac{\d y}{|y-x_0|^{d+s}} \right),\]
see \Cref{subsec:functionspacesdegiorgi} for details.\\
Further important contributions to investigation of Harnack inequalities for nonlocal operators are, among others, the articles \cite{BaLe02, ChKu03, SoVo04, BaKa05, BoSz05, CaSi07, CaSi09, BaBaChKa09,  St13, DiKuPa14, Coz17, St19, ChKuWa19} and the references therein. 

Since local minimizers of \eqref{eq:nonlocalfunctional} belong to the De Giorgi class, we have the following corollary of \Cref{thm:harnack}, that is the full Harnack inequality for local minimizers.
\begin{corollary}
Let $s_0 \in (0,1)$, $1 < p \le q$, $\Lambda \geq 1$ and assume $s \in [s_0, 1)$. Assume that $f$ satisfies \eqref{eq:pq} and let $k: \Rd \times \Rd \to \R$ be a measurable function satisfying \eqref{eq:k}. There exists a constant $C > 0$, depending only on $d$, $s_0$, $p$, $q$ and $\Lambda$, such that if $u \in V^{s,f}(\Omega | \Rd)$ is a local minimizer of \eqref{eq:nonlocalfunctional} that is nonnegative in $B_{R}(x_0) \subset \Omega$, then the full Harnack inequality \eqref{eq:harnack} holds true for $u$.
\end{corollary}
Another direct consequence of \Cref{thm:harnack}, together with the observation that weak solutions belong to the De Giorgi class, is the full Harnack inequality for weak solutions.
\begin{corollary}
Let $s_0 \in (0,1)$, $1 < p \le q$, $\Lambda \geq 1$ and assume $s \in [s_0, 1)$. Assume that $f$ satisfies \eqref{eq:pq} and let $h: \Rd \times \Rd \times \R \to \R$ be a measurable function satisfying \eqref{eq:h}. There exists a constant $C > 0$, depending only on $d$, $s_0$, $p$, $q$ and $\Lambda$, such that if $u \in V^{s,f}(\Omega | \Rd)$ is a weak solution to \eqref{eq:nonlocalequation} that is nonnegative in $B_{R}(x_0) \subset \Omega$, then  the full Harnack inequality \eqref{eq:harnack} holds true for $u$.
\end{corollary}

The proof of the main result \Cref{thm:harnack} follows from a weak Harnack inequality together with the local boundedness of functions in the De Giorgi class. Our approach roughly follows the ideas of Mascolo and Papi \cite{MascPapi96}, where they establish 
a Harnack inequality for minimizers of functionals with non-standard growth in the local case. 

We would like to point out that all results in the present paper are robust in the sense that the constants stay uniform as $s \to 1^{-}$, since they depend on $s_0$ and are independent of the actual order of differentiability $s$. Since the tail contribution vanishes as $s \to 1^{-}$, we recover purely local estimates in the limit case. Note that the nonlocal energy functional is known to converge to a local energy form as $s \to 1^{-}$, see for instance \cite{BoSa19, AlCiPi21}. 

The family of operators studied in this paper exhibits non-standard growth behavior. 
The regularity theory has been intensively studied in recent years. 
In \cite{ChKiWe21} and \cite{ByKiOk21}, local boundedness and local H\"older regularity are established for such operators using two different approaches and under slightly different conditions on the growth function. However, the condition \eqref{eq:pq} is the main assumption in both articles. For a deeper discussion on the literature about nonlocal operators with different types of non-standard growth behavior and their regularity theory, we refer the reader to the references given in those two articles. 
See also \cite{DePa19, GoDeSr20, BoSaVi20, ChKi21, Ok21, GiKuSr21b, ByOkSo21, FaZh21, GiKuSr22a} and the references therein.

Moreover, we want to point out that the local boundedness estimate in this paper, see \Cref{thm:locbdd}, significantly improves the previous result from \cite{ChKiWe21} and is proved without the condition $q<p^{\ast}$. Note that a similar result is established in \cite{ByKiOk21}.

Recently, Fang and Zhang have investigated Harnack inequalities for nonlocal operators with general growth \cite{FaZh22}. In comparison to our setup, they impose more restrictive structural assumptions on the growth function $f$. Similar to the approach in this article, they derive local boundedness and a tail estimate as in \Cref{thm:locbdd}, \Cref{lem:tail}, as well as a weak Harnack inequality. By combining these results, \cite{FaZh22} derive an upper estimate for $\sup u$ in terms of $\inf u$ and a nonlocal tail term. However, for $p < q$, this result is not optimal due to the appearance of an additional power $\iota = q/p$ in the Harnack inequality. In this article, we prove a different version of a weak Harnack inequality taking into account the growth function $f$, see \Cref{thm:inf}. This allows us to deduce a full Harnack inequality in the classical form \eqref{eq:harnack}.

\subsection*{Notation}
Throughout the paper, we will denote by $C > 0$ a universal constant, which may be different from line to line.

\subsection*{Outline}
This article is structured as follows. In \Cref{sec:preliminaries} we collect several auxiliary results for the growth functions under consideration and provide definitions of related function spaces and De Giorgi classes. \Cref{sec:locbdd} and \Cref{sec:WHI} are devoted to the proof of local boundedness and a weak Harnack inequality for functions $u$ in appropriate De Giorgi classes. Finally, the proof of the main result \Cref{thm:harnack} is provided in \Cref{sec:harnack}.

\section{Preliminaries} \label{sec:preliminaries}
This section contains several auxiliary results on the growth function $f$ and introduces the function spaces related to our setup.
\subsection{Properties of growth functions}

We collect several properties of growth functions $f: [0, \infty) \to [0, \infty)$ which were proved in \cite{ChKiWe21} and will be used in the course of this article. 
Recall that we will assume throughout this paper that $f$ is convex, strictly increasing and differentiable with $f(0) = 0$ and $f(1) = 1$.

\begin{lemma}\cite[Lemma 2.1]{ChKiWe21}
\label{lem:upper}
Let $q \geq 1$. Then the following are equivalent:
\begin{enumerate} [(i)]
\vspace{-0.2cm}
\item \eqref{eq:pq-upper},
\item $t \mapsto t^{-q}f(t)$ is decreasing,
\item $f(\lambda t) \le \lambda^q f(t)$ for all $\lambda \ge 1$,
\item $\lambda^q f(t) \le f(\lambda t)$ for all $\lambda \le 1$.
\end{enumerate}
\end{lemma}

\begin{lemma}\cite[Lemma 2.2]{ChKiWe21}
\label{lem:lower}
Let $p \geq 1$. Then the following are equivalent:
\begin{enumerate}[(i)]
\vspace{-0.2cm}
\item \eqref{eq:pq-lower},
\item $t \mapsto t^{-p}f(t)$ is increasing,
\item $\lambda^p f(t) \le f(\lambda t)$ for all $\lambda \ge 1$,
\item $f(\lambda t) \le \lambda^p f(t)$ for all $\lambda \le 1$.
\end{enumerate}
\end{lemma}

\begin{corollary}\cite[Corollary 2.4]{ChKiWe21}
Let $1 \leq p \leq q$. Assume that $f$ satisfies \eqref{eq:pq}. Then,
\begin{align}
\label{eq:derivativedoubling1}
&\frac{p}{q} \lambda^{p-1} f'(t) \leq f'(\lambda t) \leq \frac{q}{p}\lambda^{q-1} f'(t) \quad\text{for all } \lambda \geq 1, \\
\label{eq:derivativedoubling2}
&\frac{p}{q} \lambda^{q-1} f'(t) \leq f'(\lambda t) \leq \frac{q}{p}\lambda^{p-1} f'(t) \quad\text{for all } \lambda \leq 1, \\
\label{eq:der-subadd}
&\frac{1}{2}f'(t) + \frac{1}{2}f'(s) \le f'(t + s) \leq \frac{q}{p} 2^{q-1}(f'(t) + f'(s)) \quad\text{for all } t,s \geq 0.
\end{align}
\end{corollary}

\begin{lemma}\cite[Lemma 2.5]{ChKiWe21}
\label{lem:g-inv}
Let $c > 1$ and assume that for some $t,s > 0$ it holds that $f(t) \le c f(s)$. Then $t \le cs$.
\end{lemma}

Note that under the assumptions on $f$, it does not necessarily follow that $f'$ is invertible. Throughout this article, we will work with the following generalized inverse of $f'$: 
\begin{equation}\label{eq:geninverse}
(f')^{-1}(y) = \inf \lbrace t: f'(t) \geq y \rbrace.
\end{equation}

We collect a few properties of $(f')^{-1}$. First, we recall a proposition from \cite{ChKiWe21}.

\begin{proposition}\cite[Proposition 3.1]{ChKiWe21}
It holds that
\begin{align}
\label{eq:f-finv}
&(f' \circ (f')^{-1})(y) \geq y \quad \text{for all } y \geq 0,\\
\label{eq:finv-f}
&((f')^{-1} \circ f')(t) \leq t \quad \text{for all } t \geq 0.
\end{align}
\end{proposition}

The following are simple consequences of the previous results.

\begin{lemma}
\label{lem:fprimeinvadd}
For every $t,s \ge 0$:
\begin{align*}
(f')^{-1}\left( \frac{t+s}{2} \right) \le (f')^{-1}(t) + (f')^{-1}(s).
\end{align*}
\end{lemma}

\begin{proof}
By \eqref{eq:f-finv}, \eqref{eq:finv-f} and monotonicity of $(f')^{-1} $:
\begin{align*}
(f')^{-1} \left[ \frac{t+s}{2} \right] &\le (f')^{-1} \left[ \frac{f'((f')^{-1}(t))+ f'((f')^{-1}(s))}{2} \right] \le (f')^{-1} \left[ f'\left((f')^{-1}(t)+(f')^{-1}(s)\right) \right]\\
&\le (f')^{-1}(t)+(f')^{-1}(s).
\end{align*}
\end{proof}

\begin{lemma}
\label{lem:fprimeinvdoubling}
Let $1 < p \le q$. Assume that $f$ satisfies \eqref{eq:pq}. Then
\begin{align*}
(f')^{-1}(\lambda t) \le c_\lambda (f')^{-1}(t) \quad \text{for all } \lambda \geq 0,
\end{align*}
where $c_{\lambda} = (q\lambda/p)^{1/(p-1)}$ if $\lambda \geq p/q$ and $c_{\lambda} = (q\lambda/p)^{1/(q-1)}$ if $\lambda \leq p/q$.
\end{lemma}

\begin{proof}
First, we observe that by \eqref{eq:derivativedoubling1} and \eqref{eq:derivativedoubling2}, $\lambda f'(t) \le f'(c_{\lambda} t)$. Therefore, using \eqref{eq:f-finv}, \eqref{eq:finv-f} and monotonicity of $(f')^{-1} $:
\begin{align*}
(f')^{-1}\left[\lambda t\right] \le (f')^{-1}\left[\lambda f'((f')^{-1}(t))\right] \le (f')^{-1}\left[ f'(c_{\lambda}(f')^{-1}(t))\right] \le c_{\lambda}(f')^{-1}(t).
\end{align*}
\end{proof}

\subsection{Function spaces and De Giorgi classes}\label{subsec:functionspacesdegiorgi}

Let $s \in (0, 1)$ and $\Omega \subset \Rd$ be open. We define the {\it Orlicz} and {\it Orlicz--Sobolev spaces} by
\begin{equation*}
\begin{split}
L^f(\Omega) &= \lbrace u: \Omega \to \R ~\text{measurable}: \Phi_{L^f(\Omega)}(u) < \infty \rbrace, \\
W^{s, f}(\Omega) &= \lbrace u \in L^f(\Omega): \Phi_{W^{s, f}(\Omega)}(u) < \infty \rbrace,\\
V^{s, f}(\Omega|\Rd) &= \lbrace u \in L^f(\Omega): \Phi_{V^{s, f}(\Omega)}(u) < \infty \rbrace,
\end{split}
\end{equation*}
where $\Phi_{L^f(\Omega)}$, $\Phi_{W^{s, f}(\Omega)}$ and $\Phi_{V^{s,f}(\Omega|\Rd)}$ are {\it modulars} defined by
\begin{equation*}
\begin{split}
\Phi_{L^f(\Omega)}(u) &= \int_{\Omega} f(|u(x)|) \dx, \\
\Phi_{W^{s,f}(\Omega)}(u) &= (1-s)\int_{\Omega} \int_{\Omega} f\left( \frac{|u(x)-u(y)|}{|x-y|^s} \right) \frac{\dy \dx}{\vert x-y \vert^{d}} ,\\
\Phi_{V^{s,f}(\Omega|\Rd)}(u) &= (1-s)\iint_{(\Omega^c \times \Omega^c)^c} f\left( \frac{|u(x)-u(y)|}{|x-y|^s} \right) \frac{\dy \dx}{\vert x-y \vert^{d}} .
\end{split}
\end{equation*}

Next, we define nonlocal tails, which capture the behavior of functions $u\in V^{s, f}(\Omega|\Rd)$ at large scales. 
We define the nonlocal $f'$-Tail by
\begin{equation}\label{eq:tail}
\mathrm{Tail}_{f'}(u; x_0, R) = R^s (f')^{-1} \left( (1-s)R^s \int_{\Rd \setminus B_R(x_0)} f' \left( \frac{|u(y)|}{|y-x_0|^s} \right) \frac{\d y}{|y-x_0|^{d+s}} \right).
\end{equation}
Recall that the function $f'$ does not have to be invertible. Here $(f')^{-1} $ denotes the generalized inverse, see \eqref{eq:geninverse}.
In our previous work \cite{ChKiWe21}, we prove that this expression is finite if $u\in V^{s, f}(\Omega|\Rd)$ for $B_R(x_0)\subset\Omega$.

We are now ready to provide the definition of De Giorgi classes.

\begin{definition}[De Giorgi classes] \label{def:DG}
Let $\Omega$ be an open subset in $\Rd$. Let $s \in (0,1)$, $q > 1$ and $c > 0$. We say that $u \in G_+(\Omega; q, c, s, f)$ if $u \in V^{s,f}(\Omega|\Rd)$ and if for every $x_0 \in \Omega$, $0 < r < R \le d(x_0, \partial \Omega)$ and $k \in \R$, it holds that
\begin{equation}
\label{eq:DG}
\begin{split}
&\Phi_{W^{s, f}(B_r(x_0))}(w_+) + (1-s) \int_{B_{r}(x_0)} \int_{A_{k}^{-}} f'\left( \frac{w_{-}(y)}{\vert x-y \vert^s} \right) \frac{w_{+}(x)}{\vert x-y \vert^s} \frac{\d y \dx}{\vert x-y \vert^{d}} \\
&\le c\left(\frac{R}{R-r}\right)^{q} \Phi_{L^f(B_R(x_0))}\left( \frac{w_{+}}{R^s} \right) \\
&\quad +c \left(\frac{R}{R-r} \right)^{d+sq} \Vert w_+ \Vert_{L^1(B_{R}(x_0))} (1-s) \int_{\Rd \setminus B_r(x_0)} f' \left( \frac{w_+(y)}{|y-x_0|^s} \right) \frac{\d y}{|y-x_0|^{d+s}},
\end{split}
\end{equation}
where $w_{\pm}(x) = (u(x)- k)_{\pm}$ and $A_{k}^{-} = \lbrace y \in \Rd: u(y) < k \rbrace$. We say that $u \in G_-(\Omega; q, c, s, f)$ if \eqref{eq:DG} holds with $w_+$, $w_-$ and $A_{k}^{-}$ replaced by $w_-$, $w_+$ and $A_{k}^{+} = \lbrace y \in \Rd: u(y) > k \rbrace$, respectively. Moreover, we denote by $G(\Omega; q, c, s, f) = G_+(\Omega; q, c, s, f) \cap G_-(\Omega; q, c, s, f)$.
\end{definition}

\section{Local boundedness} \label{sec:locbdd}

The goal of this section is to prove local boundedness of functions $u \in G_+(\Omega; q, c, s, f)$ under \eqref{eq:pq-upper}, see \Cref{thm:locbdd}. This result significantly improves \cite[Theorem 5.1]{ChKiWe21}. Let us mention that a similar estimate has been obtained in \cite{ByKiOk21} using a different technique based on a Poincar\'e--Sobolev-type inequality for nonlocal Orlicz--Sobolev spaces. Our proof solely relies on the classical fractional Sobolev embedding and our estimate is robust for $s \to 1^{-}$.

\begin{theorem}
\label{thm:locbdd}
Let $\Omega$ be an open subset in $\Rd$. Let $0 < s_0 \leq s < 1$, $q > 1$, $c > 0$ and assume that $f$ satisfies \eqref{eq:pq-upper}. There exists a constant $C > 0$ such that if $u \in G_+(\Omega; q, c, s, f)$, then for any $B_R(x_0) \subset \Omega$, $1/2 \le \rho < \tau \leq 1$ and $\delta \in (0,1)$, it holds that
\begin{equation*}
\begin{split}
f\left( \sup_{B_{\rho R}(x_0)} \frac{u_+}{R^s} \right)
&\leq C \frac{\delta^{(1-q)2d/s_0}}{(\tau-\rho)^{\gamma}} \fint_{B_{\tau R}(x_0)} f\left( \frac{u_+(x)}{R^s} \right) \dx + \delta f \left( \frac{\mathrm{Tail}_{f'}(u_+;x_0,R/2)}{(R/2)^s} \right),
\end{split}
\end{equation*}
where $\gamma = 2d(d+q)/s_0$. The constant $C$ depends only on $d$, $s_0$, $q$ and $c$.
\end{theorem}

\begin{remark}
In particular, \Cref{thm:locbdd} implies that functions $u \in G_+(\Omega; q, c, s, f)$ are locally bounded from above in $\Omega$ under the assumptions of \Cref{thm:locbdd}. Local boundedness from below for functions $u \in G_-(\Omega; q, c, s, f)$ can be proved in the same way. Finiteness of $\mathrm{Tail}_{f'}(u_+;x_0,R/2)$ is a consequence of $u \in V^{s,f}(\Omega|\R^d)$, see \cite[Proposition 3.2]{ChKiWe21}.
\end{remark}

\begin{proof}
We may assume that $x_0=0$. For $j \geq 0$, we set
\begin{equation*}
\begin{split}
&R_j = \rho R+2^{-j}(\tau-\rho)R, \quad B_j = B_{R_j}, \\
&k_j = (1-2^{-j})k, \quad \tilde{k}_j = (k_j + k_{j+1})/2, \\
&w_j = (u-k_j)_+, \quad \tilde{w}_j = (u-\tilde{k}_j)_+,
\end{split}
\end{equation*}
where $k$ is a positive number that will be determined later. Note that $R_{j+1} < R_j \leq 2R_{j+1}$, $k_j \leq \tilde{k}_j \leq k_{j+1}$ and $w_{j+1} \leq \tilde{w}_j \leq w_j$. We denote by $A^+_{h, r}$ the set $\lbrace x \in B_r: u(x) > h \rbrace$.

Let $\sigma = \max \lbrace s_0/2, (3s-1)/2 \rbrace \in (0, s)$. Then, it is easy to check that
\begin{equation} \label{eq:sigma}
1-\sigma \leq \frac{3}{2}(1-s) \quad\text{and}\quad s-\sigma \geq \min\lbrace s_0/2, (1-\sigma)/3 \rbrace.
\end{equation}
First, by H\"older's inequality we have
\begin{equation*}
\begin{split}
\fint_{B_{j+1}} f\left( \frac{w_{j+1}(x)}{R^s} \right) \dx
&\leq \frac{1}{|B_{j+1}|} \int_{A_{\tilde{k}_j, R_{j+1}}^+} f\left(\frac{\tilde{w}_j(x)}{R^s}\right) \dx \\
&\leq \frac{|A_{\tilde{k}_j, R_{j+1}}^+|^{\frac{\sigma}{d}}}{2^{-d} |B_j|} \left( \int_{B_{j+1}} f \left( \frac{\tilde{w}_j(x)}{R^s} \right)^{\frac{d}{d-\sigma}} \dx \right)^{\frac{d-\sigma}{d}}.
\end{split}
\end{equation*}
By applying the fractional Sobolev inequality to $\tilde{w}_j/R^s$ in $B_{j+1}$, we estimate
\begin{equation*}
\begin{split}
\left( \int_{B_{j+1}} f \left( \frac{\tilde{w}_j(x)}{R^s} \right)^{\frac{d}{d-\sigma}} \dx \right)^{\frac{d-\sigma}{d}} 
&\leq C (1-\sigma) \int_{B_{j+1}} \int_{B_{j+1}} \frac{|f(\tilde{w}_j(x)/R^s) - f(\tilde{w}_j(y)/R^s)|}{|x-y|^{d+\sigma}} \dy \dx \\
&\quad + C R_{j+1}^{-\sigma} \int_{B_{j+1}} f\left( \frac{\tilde{w}_j(x)}{R^s} \right) \dx.
\end{split}
\end{equation*}
Note that from monotonicity of $f'$ and assumption \eqref{eq:pq-upper} it follows:
\begin{equation*}
\begin{split}
\frac{|f(a/R^s)-f(b/R^s)|}{|x-y|^{\sigma}}
&\leq \max \left\lbrace f'\left( \frac{a}{R^s}\right), f'\left( \frac{b}{R^s}\right) \right\rbrace \frac{|a-b|}{|x-y|^s} R^{-s} |x-y|^{s-\sigma} \\ 
&\leq q \left( \max \left\lbrace f\left( \frac{a}{R^s}\right), f\left( \frac{b}{R^s}\right) \right\rbrace + f\left( \frac{|a-b|}{|x-y|^s} \right) \right) R^{-s} |x-y|^{s-\sigma}
\end{split}
\end{equation*}
for any $a, b \in \mathbb{R}$. This inequality applied to $a=\tilde{w}_j(x)$, $b=\tilde{w}_j(y)$, together with \eqref{eq:sigma}, yields
\begin{equation*}
\begin{split}
&(1-\sigma) \int_{B_{j+1}} \int_{B_{j+1}} \frac{|f(\tilde{w}_j(x)/R^s) - f(\tilde{w}_j(y)/R^s)|}{|x-y|^{d+\sigma}} \dy \dx \\
&\leq C(1-\sigma) \int_{B_{j}} f\left( \frac{\tilde{w}_j(x)}{R^s}\right) \int_{B_{j}} \frac{R^{-s} |x-y|^{s-\sigma}}{|x-y|^{d}} \dy \dx + C R^{-s} R_{j+1}^{s-\sigma} \Phi_{W^{s, f}(B_{j+1})}(\tilde{w}_j) \\
&\leq \frac{C}{R_j^{\sigma}} \left( |B_j| \fint_{B_j} f\left( \frac{w_j(x)}{R^s} \right) \dx + \Phi_{W^{s, f}(B_{j+1})}(\tilde{w}_j) \right).
\end{split}
\end{equation*}
We have thus far obtained
\begin{equation} \label{eq:locbdd1}
\fint_{B_{j+1}} f\left( \frac{w_{j+1}(x)}{R^s} \right) \dx \leq C \left( \frac{|A_{\tilde{k}_j, R_{j+1}}^+|}{|B_j|} \right)^{\frac{\sigma}{d}} \left( \fint_{B_j} f\left( \frac{w_j(x)}{R^s} \right) \dx + \frac{\Phi_{W^{s, f}(B_{j+1})}(\tilde{w}_j)}{|B_j|}  \right),
\end{equation}
where the constant $C$ depends only on $d$, $s_0$ and $q$ at this point.\\
In order to set up a suitable iteration scheme based on \eqref{eq:locbdd1}, it remains to estimate the quantity $\Phi_{W^{s, f}(B_{j+1})}(\tilde{w}_j)$. Since $u \in G_+(\Omega; q, c, s, f)$, we have
\begin{equation} \label{eq:locbdd-I12}
\begin{split}
\Phi_{W^{s, f}(B_{j+1})}(\tilde{w}_j)
&\leq c \left( \frac{R_j}{R_j-R_{j+1}} \right)^{q} \int_{B_j} f\left( \frac{\tilde{w}_j(x)}{R_j^s} \right) \dx \\
&\quad + c \left( \frac{R_j}{R_j-R_{j+1}} \right)^{d+sq} \|\tilde{w}_j \|_{L^1(B_{j})} \int_{\Rd \setminus B_{j+1}} f' \left( \frac{\tilde{w}_j(y)}{|y|^s} \right) \frac{1-s}{|y|^{d+s}} \dy \\
&=: I_1 + I_2.
\end{split}
\end{equation}
For $I_1$, we use \Cref{lem:upper} and the fact that $R_j \leq R$ to deduce
\begin{equation} \label{eq:locbdd-I1}
I_1 \leq c \left( \frac{R_j}{R_j-R_{j+1}} \right)^{q} \left( \frac{R}{R_j} \right)^{sq} \int_{B_j} f\left( \frac{w_j(x)}{R^s} \right) \dx \leq C \frac{2^{qj}}{(\tau-\rho)^{q}} \int_{B_j} f\left( \frac{w_j(x)}{R^s} \right) \dx.
\end{equation}
For $I_2$, using monotonicity of $f'$ and the assumption \eqref{eq:pq-upper} again, we observe that
\begin{equation*}
\begin{split}
\frac{\tilde{w}_j}{R^s} f'\left( \frac{k}{R^s} \right)
&\leq q 2^{(q-1)(j+2)} \frac{\tilde{w}_j}{R^s} f'\left( \frac{\tilde{k}_j-k_j}{R^s} \right) \\
&\leq q 2^{(q-1)(j+2)} \frac{w_j}{R^s} f'\left( \frac{w_j}{R^s} \right) \leq q^2 2^{(q-1)(j+2)} f\left( \frac{w_j}{R^s} \right).
\end{split}
\end{equation*}
Thus, $I_2$ can be estimated as
\begin{equation*}
I_2 \leq C \frac{2^{(d+2q)j}}{(\tau-\rho)^{d+sq}} \left( \int_{B_j} f\left( \frac{w_j(x)}{R^s} \right) \dx \right) \frac{R^s}{f'(k/R^s)} \int_{\Rd \setminus B_{\rho R}} f'\left( \frac{u_+(y)}{|y|^s} \right) \frac{1-s}{|y|^{d+s}} \dy.
\end{equation*}
If we assume that for some $\delta \in (0,1)$, $k \geq k_1 := \delta 2^s \mathrm{Tail}_{f'}(u_+;0,R/2)$, it follows:
\begin{equation} \label{eq:locbdd-l2}
I_2 \le C \frac{\delta^{1-q}}{(\tau-\rho)^{d+sq}} 2^{(d+2q)j} \int_{B_j} f\left( \frac{w_j(x)}{R^s} \right) \dx,
\end{equation}
where we used \eqref{eq:derivativedoubling2}, and $C = C(d,s_0,q,c)$ is a positive constant. Combining \eqref{eq:locbdd-I12}, \eqref{eq:locbdd-I1} and \eqref{eq:locbdd-l2}:
\begin{equation} \label{eq:locbdd2}
\frac{1}{|B_j|} \Phi_{W^{s, f}(B_{j+1})}(\tilde{w}_j) \leq C \frac{\delta^{1-q}}{(\tau-\rho)^{d+q}} 2^{(d+2q)j} \fint_{B_j} f\left( \frac{w_j(x)}{R^s} \right) \dx.
\end{equation}

Since we have by \Cref{lem:upper}
\begin{equation} \label{eq:locbdd3}
f\left( \frac{k}{R^s} \right) \frac{|A_{\tilde{k}_j, R_{j+1}}^+|}{|B_j|} \leq q 2^{q(j+2)} f\left( \frac{\tilde{k}_j-k_j}{R^s} \right) \frac{|A_{\tilde{k}_j, R_{j}}^+|}{|B_j|} \leq q 2^{q(j+2)} \fint_{B_{j}} f\left(\frac{w_j(x)}{R^s} \right) \dx,
\end{equation}
it follows from \eqref{eq:locbdd1}, \eqref{eq:locbdd2} and \eqref{eq:locbdd3} that
\begin{equation*}
Y_{j+1} \leq C_0 \frac{\delta^{1-q}}{(\tau-\rho)^{d+q}} b^j Y_j^{1+\beta},
\end{equation*}
where $b=2^{d+3q} > 1$, $\beta=\sigma/d > 0$, $C_0=C_0(d, s_0, q, c) > 1$ and
\begin{equation*}
Y_j = \frac{1}{f(k/R^s)} \fint_{B_{j}} f\left( \frac{w_{j}(x)}{R^s} \right) \dx.
\end{equation*}
Let us take
\begin{equation*}
k = R^s f^{-1} \left( C_1 \frac{\delta^{(1-q)2d/s_0}}{(\tau-\rho)^{\gamma}} \fint_{B_{\tau R}} f\left( \frac{u_+(x)}{R^s} \right) \dx \right) + k_1,
\end{equation*}
where $\gamma = 2d(d+q)/s_0$, $C_1 = C_0^{2d/s_0} b^{4d^2/s_0^2}$ and $k_1$ is as before. This choice provides
\begin{equation*}
Y_0 \leq \left( \frac{C_0 \delta^{1-q}}{(\tau-\rho)^{d+q}} \right)^{-1/\beta} b^{-1/\beta^2},
\end{equation*}
where we used that $\sigma \ge s_0 /2$. Therefore, \cite[Lemma 4.7]{LaUr68} shows that $Y_j \to 0$ as $j \to \infty$, which concludes that $u \leq k$ a.e. in $B_{\rho R}$. By monotonicity of $f$, it follows that $f(u/R^s) \leq f(k/R^s)$ a.e. in $B_{\rho R}$, which implies the desired result due to \Cref{lem:upper} and since $f(a+b) \le 2^{q}(f(a) + f(b))$.
\end{proof}

The following result includes \Cref{thm:locbdd} as the special case $\varepsilon = 1$. It is a direct consequence of \Cref{thm:locbdd} and a classical iteration argument.

\begin{corollary} \label{cor:locbdd-e}
Let $\Omega$ be an open subset in $\Rd$. Let $0 < s_0 \leq s < 1$, $q > 1$, $c > 0$, $\varepsilon \in (0,1]$ and assume that $f$ satisfies \eqref{eq:pq-upper}. There exists a constant $C > 0$ such that if $u \in G_+(\Omega; q, c, s, f)$, then for any $B_R(x_0) \subset \Omega$ and $\delta \in (0,1)$, it holds that
\begin{equation} \label{eq:locbdd-e}
f^{\varepsilon} \left( \sup_{B_{R/2}(x_0)} \frac{u_+}{R^s} \right) \leq C \delta^{-\mu} \fint_{B_{R}(x_0)} f^{\varepsilon} \left( \frac{u_+(x)}{R^s} \right) \dx + \delta f^{\varepsilon} \left( \frac{\mathrm{Tail}_{f'}(u_+; x_0, R/2)}{(R/2)^s} \right),
\end{equation}
where $\mu = 2d(q-1)/(\varepsilon s_0)$. The constant $C$ depends only on $d$, $s_0$, $q$, $c$ and $\varepsilon$.
\end{corollary}

\begin{proof}
We may assume that $x_0 = 0$. Let $1/2 \leq \rho < \tau \leq 1$, $\delta_0 \in (0,1)$. By applying \Cref{thm:locbdd} with $\delta_0$, we have
\begin{equation} \label{eq:locbdd-e-I12}
\begin{split}
g(\rho)
&\leq C \frac{\delta_0^{(1-q)2d/s_0}}{(\tau-\rho)^{\gamma}} \fint_{B_{\tau R}} f\left( \frac{u_+(x)}{R^s} \right) \dx + \delta_0 f \left( \frac{\mathrm{Tail}_{f'}(u_+;0,R/2)}{(R/2)^s} \right) =: I_1 + I_2,
\end{split}
\end{equation}
where $C = C(d, s_0, q, c) > 0$, $\gamma = 2d(d+q)/s_0$ and
\begin{equation*}
g(\rho) = f\left( \sup_{B_{\rho R}} \frac{u_+}{R^s} \right).
\end{equation*}
Using Young's inequality, we obtain
\begin{equation} \label{eq:locbdd-e-I1}
\begin{split}
I_1
&\leq C \frac{\delta_0^{(1-q)2d/s_0}}{(\tau-\rho)^{\gamma}} g(\tau)^{1-\varepsilon} \fint_{B_{\tau R}} f^{\varepsilon} \left( \frac{u_+(x)}{R^s} \right) \dx \\
&\leq \frac{1}{2} g(\tau) + C \frac{\delta_0^{-\mu}}{(\tau-\rho)^{\gamma/\varepsilon}} \left( \fint_{B_{R}} f^{\varepsilon} \left( \frac{u_+(x)}{R^s} \right) \dx \right)^{1/\varepsilon}
\end{split}
\end{equation}
for some $C = C(d, s_0, q, c, \varepsilon) > 0$, where $\mu = 2d(q-1)/(\varepsilon s_0)$.
Combining \eqref{eq:locbdd-e-I12} and \eqref{eq:locbdd-e-I1}:
\begin{equation*}
g(\rho) \leq \frac{1}{2} g(\tau) + C \frac{\delta_0^{-\mu}}{(\tau-\rho)^{\gamma/\varepsilon}} \left( \fint_{B_{R}} f^{\varepsilon} \left( \frac{u_+(x)}{R^s} \right) \dx \right)^{1/\varepsilon} + \delta_0 f\left( \frac{\mathrm{Tail}_{f'}(u_+; 0, R/2)}{(R/2)^s} \right)
\end{equation*}
for any $1/2 \leq \rho < \tau \leq 1$. Therefore, by an iteration lemma, see \cite[Lemma 1.1]{GiGi82}:
\begin{equation*}
g(1/2) \leq C \delta_0^{-\mu} \left( \fint_{B_{R}} f^{\varepsilon} \left( \frac{u_+(x)}{R^s} \right) \dx \right)^{1/\varepsilon} + C \delta_0 f\left( \frac{\mathrm{Tail}_{f'}(u_+; 0, R/2)}{(R/2)^s} \right).
\end{equation*}
Using the inequality $(a+b)^{\varepsilon} \leq a^{\varepsilon} + b^{\varepsilon}$, we obtain
\begin{equation} \label{eq:Cdelta}
f^{\varepsilon} \left( \sup_{B_{R/2}} \frac{u_+}{R^s} \right) \leq C \delta_0^{-\varepsilon\mu} \fint_{B_{R}} f^{\varepsilon} \left( \frac{u_+(x)}{R^s} \right) \dx + C \delta_0^{\varepsilon} f^{\varepsilon} \left( \frac{\mathrm{Tail}_{f'}(u_+; 0, R/2)}{(R/2)^s} \right),
\end{equation}
where $C = C(d, s_0, q, c, \varepsilon) > 1$. For a given $\delta \in (0,1)$, the inequality \eqref{eq:locbdd-e} follows from \eqref{eq:Cdelta} by setting $\delta_0 = (\delta/C)^{1/\varepsilon} \in (0,1)$.
\end{proof}

\section{Weak Harnack inequality} \label{sec:WHI}

The goal of this section is to prove a weak Harnack inequality for functions $u \in G_-(\Omega; q, c, s, f)$. There exist several possible estimates in the literature, which go under the name ``weak Harnack inequality''. They all 
differ in the aspect that $\inf u$ is estimated by different Lebesgue-norms of $u$. We will prove an estimate of the following type since it allows us to deduce a full Harnack inequality by combination with \Cref{cor:locbdd-e}.

\begin{theorem} \label{thm:inf}
Let $\Omega$ be an open subset in $\Rd$. Let $0 < s_0 \leq s < 1$, $1<p \leq q$, $c > 0$ and assume that $f$ satisfies \eqref{eq:pq}. There exist constants $C > 0$ and $\varepsilon \in (0,1)$ such that if $u \in G_-(\Omega; q, c, s, f)$ is nonnegative in $B_{R}(x_0) \subset \Omega$, then
\begin{equation*}
\fint_{B_{R}(x_0)} f^{\varepsilon}\left( \frac{u(x)}{R^s} \right) \,\mathrm{d}x \leq C f^{\varepsilon} \left( \inf_{B_{R/2}(x_0)} \frac{u}{R^s} \right) + C f^{\varepsilon} \left( \frac{\mathrm{Tail}_{f'}(u_-; x_0, R)}{R^s} \right).
\end{equation*}
The constants $C$ and $\varepsilon$ depend only on $d$, $s_0$, $q$ and $c$.
\end{theorem}

Before we give the proof of \Cref{thm:inf}, we recall the following growth lemma from \cite{ChKiWe21}:

\begin{lemma} \cite[Theorem 4.1]{ChKiWe21}
\label{lem:growth}
Let $\Omega$ be an open subset in $\Rd$. Let $1<p\leq q$, $c, H > 0$, $R>0$, $s_0\in(0,1)$ and assume $s \in [s_0,1)$. Assume that $f$ satisfies \eqref{eq:pq}. Suppose that $B_{4R} = B_{4R}(x_0) \subset \Omega$. Let $u \in G_-(\Omega; q, c, s, g)$ satisfy $u \ge 0$ in $B_{4R}$ and
\begin{equation*}
|B_{2R} \cap \lbrace u \geq H \rbrace| \geq \gamma |B_{2R}|
\end{equation*}
for some $\gamma \in (0,1)$. There exists $\delta \in (0,1)$ such that if
\begin{equation*}
\mathrm{Tail}_{f'}(u_-; x_0, 4R) \leq \delta H,
\end{equation*}
then
\begin{equation*}
u \geq \delta H \quad\text{in}~ B_{R}.
\end{equation*}
The constant $\delta$ depends only on $d$, $s_0$, $p$, $q$, $c$ and $\gamma$.
\end{lemma}

\begin{proof}[Proof of \Cref{thm:inf}]
Without loss of generality, we assume that $x_0 = 0$.
Let us define
\begin{align*}
L := \inf_{B_{R/2}} u + \mathrm{Tail}_{f'}(u_-;0,8R).
\end{align*}
First of all, we claim that for any $H > 0$ it holds:
\begin{align}
\label{eq:inf-key}
\frac{|A^+_{t,R}|}{|B_R|} \le \left( \frac{L}{\delta t} \right)^a.
\end{align}
Here, $a = \frac{\log \frac{1}{2}}{\log \delta}$, where $\delta \in (0,1)$ is the constant from \Cref{lem:growth} applied with $\gamma = \frac{6^{-d}}{2}$ and $H := t$. 
The proof of \eqref{eq:inf-key} is a well-known consequence of \Cref{lem:growth} and a covering argument due to Krylov and Safonov. It is explained in detail in \cite[Lemma 6.7 and (6.41)]{Coz17} and can be adapted to our setup without any changes being necessary.\\
Let us explain how to deduce the desired result from \eqref{eq:inf-key}. We choose $\eps = \frac{1}{2}\min(1,\frac{a}{q})$ and compute by Cavalieri's principle and performing a change of variables
\begin{align*}
\fint_{B_R} f^{\eps}\left( \frac{u(x)}{R^s} \right) \dx &= \eps \int_{0}^{\infty} \frac{|B_R \cap \{ f(u/R^s) \ge t \}|}{|B_R|} t^{\eps - 1} \d t = \eps \int_{0}^{\infty} \frac{|A^+_{tR^s,R}|}{|B_R|} f^{\eps-1}(t) f'(t) \d t\\
&\le \eps \int_0^{L/R^s} f^{\eps-1}(t) f'(t) \d t + \eps \int_{L/R^s}^{\infty}  \left( \frac{L}{\delta t R^s} \right)^a f^{\eps-1}(t) f'(t) \d t\\
&=: I_1 + I_2.
\end{align*}
For $I_1$, by a change of variables,
\begin{align}
\label{eq:inf-I1}
I_1 = \eps \int_0^{f(L/R^s)} t^{\eps - 1} \d t = f^{\eps}\left( \frac{L}{R^s} \right).
\end{align}
For $I_2$, we apply \eqref{eq:pq-upper} and obtain
\begin{align*}
I_2 \le \eps q \delta^{-a} \left( \frac{L}{R^s} \right)^a \int_{L/R^s}^{\infty} t^{-1-a} f^{\eps}(t) \d t.
\end{align*}
From integration by parts and again \eqref{eq:pq-upper}, we see that
\begin{align*}
\int_{L/R^s}^{\infty} t^{-1-a} f^{\eps}(t) \d t &= \frac{1}{a} \left(\frac{L}{R^s}\right)^{-a} f^{\eps} \left( \frac{L}{R^s} \right) - \frac{1}{a}\lim_{t \to \infty} t^{-a}f^{\eps}(t) + \frac{\eps}{a} \int_{L/R^s}^{\infty} t^{-a} f^{\eps - 1}(t) f'(t) \d t\\
&\le \frac{1}{a} \left(\frac{L}{R^s}\right)^{-a} f^{\eps} \left( \frac{L}{R^s} \right) + \frac{1}{2} \int_{L/R^s}^{\infty} t^{-1-a} f^{\eps}(t) \d t,
\end{align*}
where we used the definition of $\eps$ and \Cref{lem:upper}(ii) in the last step.
It follows that
\begin{align}
\label{eq:inf-I2}
I_2 \le \frac{2q}{a} \delta^{-a} f^{\eps} \left( \frac{L}{R^s} \right),
\end{align}
which yields, upon combining \eqref{eq:inf-I1} and \eqref{eq:inf-I2}:
\begin{align*}
\fint_{B_R} f^{\eps}\left( \frac{u(x)}{R^s} \right) \dx \le C f^{\eps} \left( \frac{L}{R^s} \right) \le C f^{\eps} \left( \inf_{B_{R/2}} \frac{u}{R^s} \right) + C f^{\eps} \left( \frac{\mathrm{Tail}_{f'}(u_-; 0, 8R)}{R^s} \right),
\end{align*}
where we used that $f(a+b) \le 2^q(f(a) + f(b))$ and $(a+b)^{\eps} \le a^{\eps} + b^{\eps}$. The desired result follows by noticing that 
\begin{align*}
\mathrm{Tail}_{f'}(u_-; 0, 8R) \le C \mathrm{Tail}_{f'}(u_-; 0, R),
\end{align*}
which is a direct consequence of \Cref{lem:fprimeinvdoubling} applied with $\lambda = 8^s$.
\end{proof}

\section{Harnack inequality} \label{sec:harnack}

In this section, we prove our main result \Cref{thm:harnack}. First, we prove the following estimate for $\mathrm{Tail}_{f'}(u_+;x_0,R)$.

\begin{lemma} \label{lem:tail}
Let $\Omega$ be an open subset in $\Rd$. Let $0 < s_0 \leq s < 1$, $1 < p \leq q$, $c > 0$ and assume that $f$ satisfies \eqref{eq:pq}. There exists a constant $C > 0$ such that if $u \in G_-(\Omega; q, c, s, f)$ is nonnegative in $B_{R}(x_0) \subset \Omega$, then
\begin{equation*}
\mathrm{Tail}_{f'}(u_+; x_0, R/2) \le C \left(\sup_{B_{R/2}(x_0)} u + \mathrm{Tail}_{f'}(u_-; x_0, R/2)\right).
\end{equation*}
The constant $C$ depends on $d$, $s_0$, $p$, $q$ and $c$.
\end{lemma}

\begin{proof}
Without loss of generality, we may assume $x_0=0$. Let $w = u - 2M$, where $M = \sup_{B_{R/2}} u$. By $u \in G_-(\Omega;q,c,s,f)$:
\begin{equation}
\label{eq:maintail}
\begin{split}
(1-s)&\int_{B_{R/4}} w_-(x) \left( \int_{\R^d} f' \left( \frac{w_+(y)}{\vert x-y \vert^s} \right)\vert x-y \vert^{-d-s} \dy \right) \dx\\
&\le c \int_{B_{R/2}} f \left( \frac{w_-(y)}{R^s}\right) \dy + c (1-s) \Vert w_-\Vert_{L^1(B_{R/2})}\int_{B_{R/2}^c} f' \left( \frac{w_-(y)}{|y|^s} \right) |y|^{-d-s} \dy.
\end{split}
\end{equation}
Note that due to \eqref{eq:pq-lower} it holds $f'(0) = 0$. This allows us to consider the integral over $\R^d$ for the term on the left-hand side.
Since $|x-y| \le 2|y|$ for every $x \in B_{R/4}$, $y \in B_{R/2}^c$ and by \eqref{eq:der-subadd}, we estimate the first term from below by
\begin{align*}
&c(1-s)\int_{B_{R/4}} w_-(x) \left( \int_{B_{R/2}^{c}} f' \left( \frac{w_+(y)}{\vert x-y \vert^s} \right)\vert x-y \vert^{-d-s} \dy \right) \dx \\
&\ge C(1-s)M R^d \int_{B_{R/2}^{c}} f' \left( \frac{u_+(y)}{\vert y \vert^s} \right)\vert y \vert^{-d-s} \dy - C (1-s) M R^d \int_{B_{R/2}^{c}} f' \left( \frac{M}{\vert y \vert^s} \right)\vert y \vert^{-d-s} \dy.
\end{align*}

Note that by monotonicity of $f'$ and \eqref{eq:derivativedoubling1}
\begin{align*}
(1-s) M R^d \int_{B_{R/2}^{c}} f' \left( \frac{M}{\vert y \vert^s} \right)\vert y \vert^{-s-d} \dy \le C M R^{d-s} f' \left( \frac{M}{R^s} \right).
\end{align*}
Furthermore, the terms on the right-hand side of \eqref{eq:maintail} can be estimated from above by:
\begin{equation*}
C M R^{d-s} f' \left( \frac{M}{R^s} \right) + C(1-s) M R^d \int_{B_{R/2}^c} f' \left( \frac{u_-(y)}{|y|^s}  \right) |y|^{-d-s} \dy
\end{equation*}
using similar arguments. Altogether, we obtain
\begin{align*}
&(1-s) \left( \frac{R}{2} \right)^s \int_{B_{R/2}^{c}} f' \left( \frac{u_+(y)}{\vert y \vert^s} \right) |y|^{-d-s} \dy \\
&\le C f' \left( \frac{M}{R^s} \right) + C(1-s) \left( \frac{R}{2} \right)^s \int_{B_{R/2}^c} f' \left( \frac{u_-(y)}{|y|^s}  \right) |y|^{-d-s} \dy \\
&\le \frac{1}{2} \left[ f' \left( \frac{CM}{R^s} \right) + C(1-s) \left( \frac{R}{2} \right)^s \int_{B_{R/2}^c} f' \left( \frac{u_-(y)}{|y|^s}  \right) |y|^{-d-s} \dy \right],
\end{align*}
where we used \eqref{eq:derivativedoubling1} in the last step. Next, we apply $(f')^{-1}$ on both sides of the estimate and multiply with $(R/2)^s$ to obtain
\begin{align*}
\mathrm{Tail}_{f'}(u_+;0,R/2) \le C M + C \mathrm{Tail}_{f'}(u_-;0,R/2),
\end{align*}
where we applied \Cref{lem:fprimeinvadd}, \eqref{eq:finv-f} and \Cref{lem:fprimeinvdoubling}.
\end{proof}

We are now ready to give the proof of \Cref{thm:harnack}.

\begin{proof} [Proof of \Cref{thm:harnack}]
We may assume that $x_0=0$. Let $\varepsilon \in (0,1)$ be the constant from \Cref{thm:inf}. By \Cref{cor:locbdd-e} and \Cref{lem:tail}, we have
\begin{equation*}
f^{\varepsilon} \left( \sup_{B_{R/2}} \frac{u}{R^s} \right) \leq C \delta^{-\mu} \fint_{B_{R}} f^{\varepsilon} \left( \frac{u(x)}{R^s} \right) \dx + \delta f^{\varepsilon} \left( C \frac{\sup_{B_{R/2}}u + \mathrm{Tail}_{f'}(u_-; 0, R/2)}{(R/2)^s} \right),
\end{equation*}
where $\mu = 2d(q-1)/(\varepsilon s_0)$. Using \eqref{eq:pq-upper}, \Cref{lem:upper} and $u\geq 0$ in $B_R$, we obtain
\begin{equation*}
\delta f^{\varepsilon} \left( C \frac{\sup_{B_{R/2}}u + \mathrm{Tail}_{f'}(u_-; 0, R/2)}{(R/2)^s} \right) \leq C\delta \left( f^{\varepsilon} \left( \sup_{B_{R/2}} \frac{u}{R^s} \right) + f^{\varepsilon} \left( \frac{\mathrm{Tail}_{f'}(u_-; 0, R)}{R^s} \right) \right).
\end{equation*}
By taking $\delta$ sufficiently small so that $C \delta < 1/2$, we have
\begin{equation*}
f^{\varepsilon} \left( \sup_{B_{R/2}} \frac{u}{R^s} \right) \leq C \fint_{B_{R}} f^{\varepsilon} \left( \frac{u(x)}{R^s} \right) \dx + f^{\varepsilon} \left( \frac{\mathrm{Tail}_{f'}(u_-; 0, R)}{R^s} \right).
\end{equation*}
Thus, it follows from \Cref{thm:inf} that
\begin{equation*}
f^{\varepsilon} \left( \sup_{B_{R/2}} \frac{u}{R^s} \right) \leq C f^{\varepsilon} \left( \inf_{B_{R/2}} \frac{u}{R^s} \right) + Cf^{\varepsilon} \left( \frac{\mathrm{Tail}_{f'}(u_-; 0, R)}{R^s} \right).
\end{equation*}
The desired inequality follows by using $(a+b)^{\frac{1}{\varepsilon}} \leq 2^{\frac{1}{\varepsilon}-1}(a^{\frac{1}{\varepsilon}} + b^{\frac{1}{\varepsilon}})$, as well as the estimate $f(a) + f(b) \le 2 f(a+b)$ and \Cref{lem:g-inv}.
\end{proof}


\begin{thebibliography}{DCKP14}

\bibitem[ACPS21]{AlCiPi21}
Angela Alberico, Andrea Cianchi, Lubo\v{s} Pick, and Lenka Slav\'{\i}kov\'{a}.
\newblock On fractional {O}rlicz-{S}obolev spaces.
\newblock {\em Anal. Math. Phys.}, 11(2):Paper No. 84, 21, 2021.

\bibitem[BBCK09]{BaBaChKa09}
Martin~T. Barlow, Richard~F. Bass, Zhen-Qing Chen, and Moritz Kassmann.
\newblock Non-local {D}irichlet forms and symmetric jump processes.
\newblock {\em Trans. Amer. Math. Soc.}, 361(4):1963--1999, 2009.

\bibitem[BK05]{BaKa05}
Richard~F. Bass and Moritz Kassmann.
\newblock Harnack inequalities for non-local operators of variable order.
\newblock {\em Trans. Amer. Math. Soc.}, 357(2):837--850, 2005.

\bibitem[BKO21]{ByKiOk21}
Sun-Sig Byun, Hyojin Kim, and Jihoon Ok.
\newblock Local {H}{\"o}lder continuity for fractional nonlocal equations with
  general growth.
\newblock {\em arXiv:2112.13958}, 2021.

\bibitem[BL02]{BaLe02}
Richard~F. Bass and David~A. Levin.
\newblock Harnack inequalities for jump processes.
\newblock {\em Potential Anal.}, 17(4):375--388, 2002.

\bibitem[BOS21]{ByOkSo21}
Sun-Sig Byun, Jihoon Ok, and Kyeong Song.
\newblock {H}{\"o}lder regularity for weak solutions to nonlocal double phase
  problems.
\newblock {\em arXiv:2108.09623}, 2021.

\bibitem[BS05]{BoSz05}
Krzysztof Bogdan and Pawe\l{} Sztonyk.
\newblock Harnack's inequality for stable {L}\'{e}vy processes.
\newblock {\em Potential Anal.}, 22(2):133--150, 2005.

\bibitem[BSV20]{BoSaVi20}
Juli{\'a}n~Fern{\'a}ndez Bonder, Ariel Salort, and Hern{\'a}n Vivas.
\newblock Interior and up to the boundary regularity for the fractional
  $g$-{L}aplacian: the convex case.
\newblock {\em arXiv:2008.05543}, 2020.

\bibitem[CK03]{ChKu03}
Zhen-Qing Chen and Takashi Kumagai.
\newblock Heat kernel estimates for stable-like processes on {$d$}-sets.
\newblock {\em Stochastic Process. Appl.}, 108(1):27--62, 2003.

\bibitem[CK21]{ChKi21}
Jamil Chaker and Minhyun Kim.
\newblock Local regularity for nonlocal equations with variable exponents.
\newblock {\em arXiv:2107.06043}, 2021.

\bibitem[CKW19]{ChKuWa19}
Zhen-Qing Chen, Takashi Kumagai, and Jian Wang.
\newblock Elliptic {H}arnack inequalities for symmetric non-local {D}irichlet
  forms.
\newblock {\em J. Math. Pures Appl. (9)}, 125:1--42, 2019.

\bibitem[CKW21]{ChKiWe21}
Jamil Chaker, Minhyun Kim, and Marvin Weidner.
\newblock Regularity for nonlocal problems with non-standard growth.
\newblock {\em arXiv:2111.09182}, 2021.

\bibitem[Coz17]{Coz17}
Matteo Cozzi.
\newblock Regularity results and {H}arnack inequalities for minimizers and
  solutions of nonlocal problems: a unified approach via fractional {D}e
  {G}iorgi classes.
\newblock {\em J. Funct. Anal.}, 272(11):4762--4837, 2017.

\bibitem[CS07]{CaSi07}
Luis Caffarelli and Luis Silvestre.
\newblock An extension problem related to the fractional {L}aplacian.
\newblock {\em Comm. Partial Differential Equations}, 32(7-9):1245--1260, 2007.

\bibitem[CS09]{CaSi09}
Luis Caffarelli and Luis Silvestre.
\newblock Regularity theory for fully nonlinear integro-differential equations.
\newblock {\em Comm. Pure Appl. Math.}, 62(5):597--638, 2009.

\bibitem[DCKP14]{DiKuPa14}
Agnese Di~Castro, Tuomo Kuusi, and Giampiero Palatucci.
\newblock Nonlocal {H}arnack inequalities.
\newblock {\em J. Funct. Anal.}, 267(6):1807--1836, 2014.

\bibitem[DFP19]{DePa19}
Cristiana De~Filippis and Giampiero Palatucci.
\newblock H\"{o}lder regularity for nonlocal double phase equations.
\newblock {\em J. Differential Equations}, 267(1):547--586, 2019.

\bibitem[FBS19]{BoSa19}
Juli\'{a}n Fern\'{a}ndez~Bonder and Ariel~M. Salort.
\newblock Fractional order {O}rlicz-{S}obolev spaces.
\newblock {\em J. Funct. Anal.}, 277(2):333--367, 2019.

\bibitem[FZ21]{FaZh21}
Yuzhou Fang and Chao Zhang.
\newblock On weak and viscosity solutions of nonlocal double phase equations.
\newblock {\em Int. Math. Res. Not.}, 2021.

\bibitem[FZ22]{FaZh22}
Yuzhou Fang and Chao Zhang.
\newblock Harnack inequality for the nonlocal equations with general growth.
\newblock {\em arXiv:2201.09495}, 2022.

\bibitem[GG82]{GiGi82}
Mariano Giaquinta and Enrico Giusti.
\newblock On the regularity of the minima of variational integrals.
\newblock {\em Acta Math.}, 148:31--46, 1982.

\bibitem[GKS20]{GoDeSr20}
Divya Goel, Deepak Kumar, and Konijeti Sreenadh.
\newblock Regularity and multiplicity results for fractional
  {$(p,q)$}-{L}aplacian equations.
\newblock {\em Commun. Contemp. Math.}, 22(8):1950065, 37, 2020.

\bibitem[GKS21]{GiKuSr21b}
Jacques Giacomoni, Deepak Kumar, and Konijeti Sreenadh.
\newblock Interior and boundary regularity results for strongly nonhomogeneous
  $p$, $q$-fractional problems.
\newblock {\em Adv. Calc. Var.}, 2021.

\bibitem[GKS22]{GiKuSr22a}
Jacques Giacomoni, Deepak Kumar, and Konijeti Sreenadh.
\newblock Global regularity results for non-homogeneous growth fractional
  problems.
\newblock {\em J. Geom. Anal.}, 32(1):Paper No. 36, 41, 2022.

\bibitem[Kas07a]{Ka07}
Moritz Kassmann.
\newblock The classical {H}arnack inequality fails for non-local operators.
\newblock {\em SFB 611 - Preprint, No. 360.}, 2007.

\bibitem[Kas07b]{KassmannHarnack}
Moritz Kassmann.
\newblock Harnack inequalities: an introduction.
\newblock {\em Bound. Value Probl.}, pages Art. ID 81415, 21, 2007.

\bibitem[Kas11]{Ka11}
Moritz Kassmann.
\newblock A new formulation of {H}arnack's inequality for nonlocal operators.
\newblock {\em C. R. Math. Acad. Sci. Paris}, 349(11-12):637--640, 2011.

\bibitem[LU68]{LaUr68}
Olga~A. Ladyzhenskaya and Nina~N. Ural'tseva.
\newblock {\em Linear and quasilinear elliptic equations}.
\newblock Academic Press, New York-London, 1968.
\newblock Translated from the Russian by Scripta Technica, Inc, Translation
  editor: Leon Ehrenpreis.

\bibitem[MP96]{MascPapi96}
Elvira Mascolo and Gloria Papi.
\newblock Harnack inequality for minimizers of integral functionals with
  general growth conditions.
\newblock {\em NoDEA Nonlinear Differential Equations Appl.}, 3(2):231--244,
  1996.

\bibitem[Ok21]{Ok21}
Jihoon Ok.
\newblock Local {H}{\"o}lder regularity for nonlocal equations with variable
  powers.
\newblock {\em arXiv:2107.06611}, 2021.

\bibitem[Str19]{St19}
Martin Str\"{o}mqvist.
\newblock Harnack's inequality for parabolic nonlocal equations.
\newblock {\em Ann. Inst. H. Poincar\'{e} Anal. Non Lin\'{e}aire},
  36(6):1709--1745, 2019.

\bibitem[SV04]{SoVo04}
Renming Song and Zoran Vondra\v{c}ek.
\newblock Harnack inequality for some classes of {M}arkov processes.
\newblock {\em Math. Z.}, 246(1-2):177--202, 2004.

\bibitem[SZ13]{St13}
Pablo~Ra\'{u}l Stinga and Chao Zhang.
\newblock Harnack's inequality for fractional nonlocal equations.
\newblock {\em Discrete Contin. Dyn. Syst.}, 33(7):3153--3170, 2013.

\end{thebibliography}

\end{document}